\documentclass[12pt]{article}

\usepackage[margin=1in]{geometry}
\usepackage{amscd}
\usepackage{amssymb}
\usepackage{amsmath}
\usepackage{amsthm}
\usepackage{enumerate}
\usepackage{tikz}
\usepackage{wrapfig, framed, caption}
\usepackage{float}
\usetikzlibrary{arrows}
\usepackage[font=small,labelfont=bf]{caption}
\usepackage{xcolor}
\usepackage{mathtools}
\usepackage[colorlinks=true, linkcolor=blue, citecolor=blue,
pagebackref=true]{hyperref}
\usepackage{fouriernc}
\usepackage{ulem}

\usepackage{comment}

\newtheorem{theorem}{Theorem}
\newtheorem*{theorem*}{Theorem}

\newtheorem*{theoremY*}{Theorem Y}

\newtheorem*{theoremAB*}{Theorem AB}

\newtheorem*{linearformsmtp*}{Mass transference principle for linear forms}

\newtheorem*{corollary*}{Corollary}

\newtheorem{lemma}{Lemma}

\newtheorem*{claim*}{Claim}

\theoremstyle{definition}

\theoremstyle{remark}
\newtheorem{remark}{Remark}
\newtheorem*{remark*}{Remark}

\renewcommand{\Bbb}[1]{\mathbb{#1}}

\newcommand{\N}{{\Bbb N}}         

\newcommand{\Q}{{\Bbb Q}}         
\newcommand{\R}{{\Bbb R}}        

\newcommand{\W}{\mathcal{W}}

\newcommand{\Z}{{\Bbb Z}}         

\newcommand{\cC}{{\cal C}}
\newcommand{\cD}{{\cal D}}

\newcommand{\cF}{{\cal F}}
\newcommand{\cG}{{\cal G}}
\newcommand{\cH}{{\cal H}}

\newcommand{\cL}{{\cal L}}

\newcommand{\cS}{{\cal S}}

\newcommand{\x}{\mathbf{x}}

\newcommand{\dist}{\operatorname{dist}}

\renewcommand{\le}{\leq}

\newcommand{\Exact}{\mathbf{Exact}}
\newcommand{\Bad}{\mathbf{Bad}}

\DeclarePairedDelimiter{\norm}{\lVert}{\rVert}

\DeclareMathOperator{\dimh}{\dim_H}

\title{The dimension of the set of $\psi$-badly approximable points in all ambient dimensions; on a question of Beresnevich and Velani. }

\author{Henna Koivusalo \\ (Bristol) \and Jason Levesley \\ (York) \and Benjamin Ward\footnote{B. W. gratefully acknowledges support from the EPSRC research grant (EP/W522430/1), and Australian Research Council Discovery grant (No. 200100994).} \\ (La Trobe) \and Xintian Zhang \\ (Bristol)}

\date{\today}
\begin{document}
\frenchspacing
\maketitle
%
%

%
%

\begin{abstract}

Let $\psi:\N \to [0,\infty)$, $\psi(q)=q^{-(1+\tau)}$ and let $\psi$-badly approximable points be those vectors in $\R^{d}$ that are $\psi$-well approximable, but not $c\psi$-well approximable for arbitrarily small constants $c>0$. We establish that the $\psi$-badly approximable points have the Hausdorff dimension of the $\psi$-well approximable points, the dimension taking the value $(d+1)/(\tau+1)$ familiar from theorems of Besicovitch and Jarn\'ik. The method of proof is an entirely new take on the Mass Transference Principle by Beresnevich and Velani (Annals, 2006); namely, we use the colloquially named `delayed pruning' to construct a sufficiently large $\liminf$ set and combine this with ideas inspired by the proof of the Mass Transference Principle to find a large $\limsup$ subset of the $\liminf$ set. Our results are a generalisation of some  $1$-dimensional results due to  Bugeaud and Moreira (Acta Arith, 2011), but our method of proof is nothing alike. 
\end{abstract}

%
%

\section{Introduction} \label{intro}
Throughout let $\psi:\N \to (0, \infty)$ be a monotonic decreasing function. We will use the notation $\psi_\tau$ when $\psi(q) = q^{-\tau}$ for $\tau \in (0, \infty)$. The set of  {\it $\psi$-well approximable points}, denoted by $\W_{d}(\psi)$, is defined to be the set
\begin{equation*}
\W_{d}(\psi):=\left\{x \in [0,1]^{d} : \max_{1 \leq i \leq d} \left|x_{i}-\frac{p_{i}}{q} \right| < \frac{\psi(q)}{q} \quad \text{for i.m.} \, \, \, (p,q) \in \Z^{d} \times \N \right\}\, ,
\end{equation*}
where i.m. denotes infinitely many. The main object of interest in this article is the {\it Hausdorff dimension} of the {\it $\psi$-badly approximable points}; those points of $\W_d(\psi)$ for which the approximation by $\psi$ cannot be improved by an arbitrarily small constant. More precisely,
\[
    \Bad_{d}(\psi):= 
    \W_{d}(\psi) \backslash \bigcap_{k=1}^{\infty} \W_{d}\left(\frac{1}{k}\psi \right).
\]
In \cite[Question 1]{BV08} Beresnevich and Velani asked for the Hausdorff dimension of $\Bad_{d}(\psi)$ and conjectured that the Hausdorff dimension was equal to the dimension of $\W_{d}(\psi)$. We prove this conjecture true for a certain class of approximation functions. \par 
We recall the definition of Hausdorff measure and dimension. (For more details on such notions see for example \cite{Rog98, F14}.) For $ s\geq 0$ the {\it Hausdorff $s$-measure} of a set $F \subset \R^{d}$ is defined as
\begin{equation*}
    \cH^{s}(F):=\lim_{\rho \to 0^{+}} \inf \left\{ \sum_{i=1}^\infty |B_{i}|^{s} : F \subset \bigcup_{i=1}^\infty  B_{i} \quad \text{ and }\quad  |B_{i}|<\rho \, \, \forall \, \, i \right\} \, 
\end{equation*}
where $|\cdot |$ denotes the diameter of a set. The {\it Hausdorff dimension} of set $F \subset \R^{d}$ is defined as
\begin{equation*}
    \dimh F := \inf \left\{ s \geq 0 : \cH^{s}(F)=0 \right\}\, .
\end{equation*}
In the case when $ s = d $, the ambient dimension of our setting, Hausdorff $s$-measure coincides (up to a constant) with Lebesgue  measure which we denote by $\lambda_d$. The literature surrounding the metric properties of $\Bad_d(\psi)$ is extensive, and we only highlight some of the main results below. 

For approximation function $\psi_{1/d}$, the set $\Bad_{d}(\psi_{1/d})=\Bad_{d}$ is the classical set of badly approximable points. As a consequence of a landmark theorem of Khintchine \cite{K24} 
 we know that $\lambda_{d}(\Bad_d) = 0$ (This can be proven by a range of methods, see \cite{BDGW23} for an overview of techniques). It is then natural to ask, what is the Hausdorff dimension of $\Bad_d$. Jarn\'{\i}k \cite{J28} and Schmidt \cite{Schmidt66} showed that despite the set being null, it has full Hausdorff dimension in the cases where $d=1$ and $d>1$, respectively. That is
 \begin{equation*}
     \dimh \Bad_{d}=d\, .
 \end{equation*}

In the direction of $\psi$-well approximable points, Besicovitch \cite{B34} and Jarn\'{\i}k \cite{J29} independently proved that whenever $\tau>\frac{1}{d}$ 
\begin{equation*}
    \dimh \W_{d}(\psi_{\tau}) = \frac{d+1}{\tau+1} \,.
\end{equation*}
Dodson extended this result to the case of general approximation functions $\psi$ \cite[Theorem 2]{Dod92}. As $\Bad_{d}(\psi) \subset \W_{d}(\psi)$, these results give immediate upper bounds for the Hausdorff dimension of $\Bad_{d}(\psi)$. As is often the case, establishing the corresponding lower bound is the main obstacle in determining the exact dimension. 

The following refinement of $\W_d(\psi)$ and $\Bad_d(\psi)$, first introduced in \cite{BDV01}, is useful as they can be used to define sets whose points satisfy more delicate approximation properties. Let $\psi, \phi : \N \to (0, \infty)$ be monotonic decreasing functions with $\psi(q)>\phi(q)$ for all $q \in \N$ and denote
\begin{align*}
    D_{d}(\psi, \phi)
   & := \left\{ x \in [0,1]^{d}: \underset{\stackrel{\text{for all $(p,q) \in \Z^{d} \times \N$}}{\text{with $q$ sufficiently large}}}{\underbrace{\, \quad \phi(q)/q \, \, \leq \quad \, }}\max_{1 \leq i \leq d}\left|x_{i}-\frac{p_{i}}{q}\right| \underset{\text{for i.m. $(p,q) \in \Z^{d} \times \N$ }}{\underbrace{\, \quad < \, \,  \psi(q)/q \quad \, }} \right\} \\
    &=\W_{d}(\psi) \backslash \W_{d}(\phi).
\end{align*}
With these sets defined, 
the set of $\psi$-badly approximable points can now be alternatively written as
\[
    \Bad_{d}(\psi)= \bigcup_{k \in \N} D_{d}\left(\psi, \frac{1}{k}\psi\right),
\]
and the set of {\it exact approximation order $\psi$} is
\[
    \Exact_{d}(\psi):= \bigcap_{k \in \N} D_{d}\left(\psi, \left(1-\frac{1}{k}\right)\psi \right) = 
    \W_{d}(\psi) \backslash \bigcup_{k=1}^{\infty} \W_{d}\left(\left(1-\frac{1}{k}\right)\psi \right).
\]
Trivially $\Exact_{d}(\psi) \subset \Bad_{d}(\psi) \subset \W_{d}(\psi)$ and so $\Exact_{d}(\psi)$ can be used for lower bounds of dimension. 

One of the first results on the set $\Exact_{1}(\psi)$ was given by Jarn\'{\i}k \cite{Jar31} who, using the theory of continued fractions, proved that $\Exact_{1}(\psi) \not= \emptyset$ for functions $\psi(q)=o(q^{-1})$. Without the condition $\psi(q)=o(q^{-1})$ the theory becomes significantly more complicated. In particular, a classical result of Hurwitz tells us that $\W_{1}\left(\frac{1}{\sqrt{5}}\psi_{1}\right)=[0,1]$ and so
\begin{equation*}
\Exact_{1}(\psi_{1}) \subset \W_{1}(\psi_{1}) \backslash \W_{1}\left(\frac{1}{\sqrt{5}}\psi_{1}\right) = \emptyset\, .
\end{equation*}
 Conversely, if $\psi(q)=cq^{-1}$ for some $c>0$ small it is not necessarily true that $\Exact_{1}(\psi)=\emptyset$. In fact, for $c\to 0$, Moreira showed that the Hausdorff dimension of $\Exact_{1}(\psi)$ tends to one \cite[Theorem 2]{Mor18}. Nevertheless, throughout we will suppose the approximation function satisfies $\psi(q)=o\left(q^{-1/d}\right)$.

 Since the result of Jarn\'{\i}k there has been gradual progress towards establishing the Hausdorff dimension of $\Exact_{1}(\psi)$. In \cite{G63} G\"{u}ting 
proved that 
\[
    \dimh \left( \bigcup_{k \in \N} D_{1}\left(\psi_{\tau}, \psi_{\tau+\frac{1}{k}}\right) \right)=\dimh \W_{1}(\psi_{\tau})=\frac{2}{1+\tau},
\]
for $\tau > 1$.
In \cite{BDV01} Beresnevich, Dickinson and Velani improved upon G\"{u}ting's result by calculating the Hausdorff measure at the dimension and showed that
\[
    \cH^{\frac{2}{1+\tau}}\left( \bigcup_{k \in \N} D_{1}\left(\psi_{\tau}, \psi_{\tau+\frac{1}{k}}\right) \right)= \infty \, .
\]
Furthermore they considered approximation functions with "logarithmic error" and showed that for $\tau>\frac{1}{d}$
\begin{equation*}
    \dimh \left( \bigcup_{k \in \N} D_{d}\left( \psi_{\tau},\, q \mapsto q^{-\tau}(\log q)^{-\frac{1}{k}}\right) \right)=\frac{d+1}{\tau+1}. 
\end{equation*}
This was established by comparing the Hausdorff measures of the two approximation sets $\W_{d}( \psi_{\tau})$ and $\W_{d}\left(q \mapsto q^{-\tau}(\log q)^{-\varepsilon}\right)$ for small $\varepsilon>0$. However the technique used here cannot be applied in the "constant error" case as the Hausdorff measure of the two sets would be the same. In a series of papers Bugeaud \cite{Bu03, Bu08} and Bugeaud and Moreira \cite{BM11} proved a complete result in one dimension\footnote{Indeed, this result of Bugeaud and Moreira doesn't only hold for $\psi_\tau$ but also for general approximation functions $\psi$. We state this special case for simplicity as it is the relevant statement for the purposes of the current work.}, showing that
\begin{equation*}
    \dimh \Exact_{1}(\psi_\tau)=\frac{2}{1+\tau}. 
\end{equation*}
It is then a consequence of the inclusions $\W_{1}(\psi) \supset \Bad_{1}(\psi) \supset \Exact_{1}(\psi)$ that 
\begin{equation*}
    \dimh \Bad_{1}(\psi_\tau)=\frac{2}{1+\tau} \, ,
\end{equation*}
Furthermore, in \cite[Theorem 2]{Bu03} Bugeaud also showed that
\begin{equation*}
    \cH^{\frac{2}{1+\tau}}(\Bad_{1}(\psi_{\tau}))=\infty\,.
\end{equation*}
The series of papers by Bugeaud and Moreira rely on results from the theory of continued fractions. Higher dimensional versions of continued fractions have been extensively studied, see for example \cite{Brentjes1981,Schweiger2000} and \cite{BDKKKT2023, Karpenkov2022} for more recent approaches. However, trying to apply arguments from $1$-dimension to higher dimensions often fails and in particular, the methods of Bugeaud and Moreira do not seem to be applicable to the study of $\Bad_d(\psi_\tau)$.

In the current work we compute the Hausdorff dimension $\Bad_{d}(\psi_{\tau})$ for any dimension $d$. That is, we prove the following;

\begin{theorem} \label{main}
Let $\psi_{\tau}(q)=q^{-\tau}$ with $\tau>\frac{1}{d}$, and let $B\subset [0,1]^{d}$ be any open ball. Then  
\[
\dimh \Bad_{d}(\psi_{\tau})\cap B = \frac{d+1}{1+\tau}.
\]

\end{theorem}

The proof has been inspired by the proof of the Mass Transference Principle (MTP) of Beresnevich and Velani \cite{BV06}. The MTP is a powerful technique, introduced to Diophantine approximation in 2006, when it was used to prove a Hausdorff measure version of the Duffin-Schaeffer conjecture. The original MTP states, loosely speaking, that if a limsup set of balls has full Lebesgue measure, then the limsup set of the same balls shrunk in a controlled way satisfies a dimension lower bound. The MTP has since been adapted, e.g. for other measures, and other shapes \cite{AB19, KR21, WW19, WWX15}. However, in studying $\psi$-badly approximable points, we are facing a set-up not covered by the many generalisations: we need a Hausdorff dimension estimate for the $\limsup$ set given by balls centred outside of the zero measured set in which we would like the dimension lower bounds to hold! The proof we present relies on a few key ideas: We emulate $\Bad_d(\psi_\tau)$ by a set of almost full measure (see \S~ \ref{cC(N)}), and for `positive' approximation we use balls centred at rationals that are not too far away from this set (see \S~\ref{subset_W(psi)}). These allow us to approach the dimension bound in the spirit of the proof of MTP, although there are some further geometric obstacles to overcome (see \S~ \ref{cD(B)}).

%

Our main result is actually a consequence of the following result, which we also state for comparison to the series of results of Moreira and Bugeaud on dimension of $\Exact_1(\psi_\tau)$. 

\begin{theorem} \label{D-set}
    Let $\psi_{\tau}$ and $B\subset [0,1]^{d}$ be as above with $\tau>\frac{1}{d}$. Then there exists constant $0<C<1$ dependent only on $B$, $\tau$ and $d$ such that 
    \begin{equation*}
        \dimh D_{d}\left(\psi_{\tau},C\psi_{\tau} \right)\cap B =\frac{d+1}{1+\tau}.
    \end{equation*}
\end{theorem}

\begin{remark}
    Note that Theorem~\ref{D-set} immediately implies Theorem~\ref{main}, since $D_{d}\left(\psi_{\tau},C\psi_{\tau} \right) \subset \Bad_{d}(\psi_{\tau})$. The constant $C>0$ is implicit in our proofs and hasn't been optimised. However, we cannot make it arbitrary close to $1$, and so we cannot prove any result on $\Exact_{d}(\psi_{\tau})$. Very recently Fregoli has proven a restricted higher dimensional version of the results of Bugeaud and Moreira \cite{Bu03, Bu08, BM11} for $n\geq 3$ and $\tau>1$. The technique used by Fregoli took their results and lifted it to higher dimensions by a clever observation. Briefly, for $\boldsymbol{x}=(x_{1},\boldsymbol{y})\in\R^{d}$ if $x_{1} \in \Exact_{1}(\psi_{\tau})$ and $\boldsymbol{y} \in W_{d-1}(\psi_{\tau})$ then $\boldsymbol{x} \in \Exact_{d}(\psi_{\tau})$. The dimension result can then be proven via the result of Bugeaud and Moreira and a Theorem of Jarn\'{\i}k on fibres. See \cite{Fregoli2023} for more details, we only wish to point out that our technique is significantly different. For other notions of size it was recently shown by Schleischitz that for certain functions $\psi$ the set $\Exact_{d}(\psi)$ has full packing dimension, see \cite[Theorem 3.6 \& Corollary 7]{Schleischitz2023}.
\end{remark}
\vspace{.5ex}
\begin{remark} 
For this set of approximation functions the restriction that $\tau>\frac{1}{d}$ is appropriate. If $\tau=\frac{1}{d}$ then as previously mentioned $\Bad_{d}(\psi_{\tau})=\Bad_{d}$, that is, the classical set of $d$-dimensional badly approximable points which has full Hausdorff dimension \cite{Schmidt66}. If $\tau<\frac{1}{d}$ then $\Bad_{d}(\psi_{\tau})=\emptyset$ since $\W_{d}(c \psi_{\tau}) \supseteq \W_{d}(\psi_{1/d})=[0,1]^{d}$ for any constant $c>0$ and so $\Bad_{d}(\psi_{\tau})\subset W_{d}(\psi_{\tau}) \backslash [0,1]^{d}= \emptyset$. 
\end{remark} 
\vspace{.5ex}
\begin{remark} 
In the current work we haven't considered general approximation functions $\psi$. Analogous to the results of Bugeaud and Moreira \cite{Bu03, Bu08, BM11}, we suspect that for monotonic decreasing $\psi(q)$ of order $o(q^{-\frac{1}{d}})$ 
\begin{equation*}
\dimh \Bad_{d}(\psi)=\dimh\W_d(\psi), 
\end{equation*}
where the Hausdorff dimension of $\dimh \W_d(\psi)$ is known \cite{Dod92}. We believe that it may be possible to adapt our argument to deduce this result in the case $\sum_{q=1}^{\infty} \psi(q)^{d}<\infty$. The case where $\sum_{q=1}^{\infty} \psi(q)^{d}=\infty$ appears significantly harder. 
\end{remark} 
\vspace{.5ex}

\begin{remark}
The theory of exact approximation has also been studied in a range of other settings. ZL Zhang proved a version of Bugeaud and Moreiras' result in the setting of the field of formal series \cite{Z12}, and He \& Xiong \cite{HX22} have proven an analogue in the setting of approximation by complex rational numbers. Both papers appeal to results from the theory of continued fractions in their respective setting to prove their main theorem. Very recently Bandi, Ghosh and Nandi \cite{BGN22} have proven a general theory in the setting of exact approximation. In their paper they obtain a result on exact approximation sets in hyperbolic metric spaces. However, as far as we can see the $d$-dimensional real result presented here cannot be deduced from the general theory presented in their paper. \par
\end{remark}

 The rest of the paper is organised as follows. In the following section we give an outline of the proof of Theorem \ref{main}. In \S~\ref{cC(N)} we find a subset $\cC^\tau(N) \subset \Bad_d(\psi_\tau)$ of large measure, as quantified in \S~\ref{cC(N)_measure}. In \S~\ref{subset_W(psi)} we define the subset of rationals $Q(N, \tau)$ that we consider for positive approximation on $\cC^\tau(N)$, and show that this subcollection of rationals is not too small. In \S~\ref{cD(B)} we present a construction of a Cantor set and a mass distribution, following the general outline of the proof of Mass Transference Principle, but relying on a new geometric approximation Lemma \ref{TGI}.

%

\section{Outline of the proof of Theorem~\ref{main}} \label{proof_main}

Before going into the proof of Theorem~\ref{main} we give an overview of the methodology used. 

Firstly, since $\Bad_{d}(\psi_{\tau}) \subseteq \W_{d}(\psi_{\tau})$ the upper bound  follows from the Jarn\'{\i}k-Besicovitch Theorem. The main substance of the proof is in establishing a corresponding lower bound. 
To accomplish this we construct a Cantor subset $\cD(B)$ of $ \Bad_{d}(\psi_{\tau}) \cap B $ where $B$ is any arbitrary open ball in $(0,1)^{d}$. 
The construction of $\cD(B)$ can be split into three main steps. 

{\bf Step 1:} We construct a set that avoids `dangerous balls' i.e. parts of the unit cube that lie `too close' to rational points. The name `dangerous balls' is taken from the analogous set that is removed when constructing a Cantor subset of badly approximable points, see for example \cite[\S 7]{BRV16}. In this setting, given $N \in \N$ and $\tau> \frac{1}{d}$, we construct a set $\cC^{\tau}(N)$ such that
\begin{equation*}
\cC^{\tau}(N) \cap \bigcup_{\frac{p}{q} \in \Q^{d} \cap [0,1]^{d} } B\left(\frac{p}{q}, c_{N}q^{-1-\tau} \right)= \emptyset \, ,
\end{equation*}
for some constant $c_{N}>0$, where $B(x,r)=\{y \in \R^{d}: \norm{x-y}<r\}$ denotes the open ball with centre $x \in \R^{d}$ of radius $r>0$ and $\| \cdot \| $ is the standard supremum norm in $\mathbb{R}^d$. This set, and properties of $\cC^{\tau}(N)$ that will be required later in the proof of the main result, are dealt with in \S\ref{cC(N)}.  
 
{\bf Step 2:} We then construct a $\limsup$ subset of $\W_{d}(\psi_{\tau})$ such that each ball in the $\limsup$ set retains a positive proportion of its mass when intersected with $\cC^{\tau}(N)$. We achieve this by constructing a set $Q(N,\tau) \subset \Q^{d}$ such that for any $\frac{p}{q} \in Q(N,\tau)$ the inequality 
 \[
 \lambda_{d}\left( \cC^{\tau}(N) \cap B\left(\frac{p}{q}, q^{-1-\tau} \right) \right) \geq \kappa \lambda_{d}\left( B\left(\frac{p}{q}, q^{-1-\tau} \right)\right)\, 
 \]
 holds for some fixed constant $\kappa>0$ independent of $\frac{p}{q}$. This is established in the course of proving Lemma~\ref{Q_approximations}.  

{\bf Step 3:} The final part of the argument is to construct a Cantor subset and a mass distribution, inspired by the Mass Transference Principle (a now standard tool for proving metric results on $\limsup$ sets, see \cite{AT19}). A key step in the construction of the Cantor set $\cD(B)$ is to show that for every ball $\tilde{B}=B\left(\frac{p}{q}, q^{-1-\tau} \right)$ with $\frac{p}{q} \in \Q$ there exists a finite disjoint collection of balls of the form $B\left(\frac{p'}{q'}, q'^{-1-\tau} \right)$ with $\frac{p'}{q'} \in \Q$ that are contained in $\tilde{B}$ and intersect a positive proportion of $\cC^{\tau}(N)$. A crucial lemma to ensure this happens is Lemma~\ref{TGI} of \S\ref{cD(B)}. This statement is the counterpart of the $K_{G,B}$ lemma from the proof of the MTP \cite{BV06}. 

 This methodology is sufficient to prove Theorem~\ref{D-set} which in turn implies Theorem~\ref{main}.

%
%

\section{Step 1: The construction of \texorpdfstring{$\cC^{\tau}(N)$}{the first Cantor set}} \label{cC(N)}
\texorpdfstring{$\cC^{\tau}(N)$}{TEXT}
 The following `Simplex lemma' is a crucial element in the construction of $\cC^{\tau}(N)$. The proof can be found in \cite{KTV06}.
\begin{lemma}\label{simplex_lemma} Let $d \geq 1$ be an integer and $Q \in \N$. Let $E \subset \R^{d}$ be a convex set of $d$-dimensional Lebesgue measure
\begin{equation*}
\lambda_{d}(E) \leq (d!)^{-1}Q^{-(d+1)}.
\end{equation*}
Suppose $E$ contains $d+1$ rational points with denominator $1 \leq q \leq Q$. Then these rational points lie on some hyperplane of $\R^{d}$.
\end{lemma}
By splitting up the construction of $\cC^{\tau}(N)$ into suitable levels, this result tells us that regions containing `dangerous balls' will look like thickened $(d-1)$-dimensional hyperplanes in $[0,1]^{d}$. 
 
Fix some $N \in \N$. (We will need to take this $N$ large enough along the course of the proofs, but it will eventually be fixed.) Construct $\cC^{\tau}(N)$ as follows:
\begin{enumerate}
 \item 
Split $[0,1]^{d}$ into $t^{d}$ cubes of sidelength $t^{-1}$ for some $t \in \N$ where $t$ is chosen to satisfy $t > (d!)^{1/d}$. Now split each of these cubes into $N^{d+1}$ cubes $I$ each with side-length $t^{-1}N^{-\frac{d+1}{d}}$ and volume $t^{-d}N^{-(d+1)}$. Let $C_{1}$ denote the set of cubes constructed in this way. Note there are $t^{d}N^{(d+1)}$ cubes in total. 
\item Within each cube $I$ we want to eventually remove the rational points $(\frac{p_{1}}{q}, \dots , \frac{p_{d}}{q})$ with $q$ bounded by
 $1 \leq q < N$.
  Apply Lemma~\ref{simplex_lemma} to each cube in $I \in C_{1}$ to establish that any such rational points are contained in some hyperplane $L$ which we  choose to be the minimal such affine subspace. 
That is, if there is only a single point $\frac{p}{q} \in I$ then $L$ is a point, if there are two points $\frac{p}{q}, \frac{p'}{q'} \in I$ then $L$ is the unique line between $\frac{p}{q}$ and $\frac{p'}{q'}$ intersected with $I$ and so on. Lemma~\ref{simplex_lemma} tells us that at worst $L$ could be a $d$-dimensional affine hyperplane. For simplicity we will call any such $L$ a hyperplane. 
 Let $\cL_{1}$ denote the collection of all hyperplanes $L$ constructed in this manner for each cube in $C_{1}$. Remember these sets for later.
 \item Repeat this process; Assume that levels up to $n-1$ have been constructed. At the $n$th level split each cube $I \in C_{n-1}$ into $N^{d+1}$ cubes each with
side length $t^{-1} N^{-(n-1)\frac{d+1}{d}} \times N^{-\frac{d+1}{d}}=t^{-1} N^{-n\frac{d+1}{d}}$  and volume $t^{-d} N^{-n(d+1)}$.
 Let $C_{n}$ denote all cube constructed in this layer. For a cube $I \in C_{j}$ for $j=1, \dots , n-1$ let $C_{n}(I)$ denote all cubes in $C_n$  which are contained in $I$. 
 Apply Lemma~\ref{simplex_lemma} to each $I^\prime \in C_{n}$ to find a hyperplane $L$ containing all rational points $(\frac{p_{1}}{q}, \dots , \frac{p_{d}}{q}) \in I^\prime$ with 
 $1 \leq q < N^{n}$.
 Let $\cL_{n}$ denote the collection of all level $n$ hyperplanes $L$ constructed in this way. For a cube $I \in C_{j}$ for $j=1, \dots , n-1$, let $\cL_{n}(I)$ denote the set of all hyperplanes that correspond to cubes $I^\prime \in C_{n}(I)$. Observe that a hyperplane $L \in \cL_{n}$ contains all the rational points with denominators 
 \begin{equation*}
 N^{n-1} \leq q < N^{n}. 
 \end{equation*}

 \item \textit{Removal of the hyperplanes}: Recall we eventually want to remove all dangerous balls $B\left(\tfrac{p}{q}, cq^{-(1+\tau)}\right)$ for some constant $c>0$. Choose
 \begin{equation*}
 c_{N}=t^{-1}N^{-u(1+\tau)}
 \end{equation*}
 with $u \in \N$ constant such that
 \begin{equation}\label{u_def}
     u>3>2+\frac{(d+1)(d-1)}{(1+\tau)d^{2}}.
 \end{equation}
 Then
 \begin{equation*}
 \bigcup_{N^{n-1} \leq q < N^{n}} B\left(\frac{p}{q} , c_{N}q^{-(1+\tau)}\right) \subset \bigcup_{N^{n-1} \leq q < N^{n}} B\left(\frac{p}{q}, t^{-1}N^{-(n-1+u)(1+\tau)}\right).
 \end{equation*}
 At layer 
 \begin{equation*}
     \ell(1):=\left\lfloor u\frac{(1+\tau)d}{d+1}\right\rfloor
 \end{equation*}
 remove all cubes $I\in C_{\ell(1)}$ that intersect the $t^{-1}N^{-u(1+\tau)}$-thickening of all hyperplanes contained in $\cL_{1}$. That is, for a hyperplane $L \in \cL_{1}$ constructed in the cube $I' \in C_{1}$ and
 \begin{equation*}
     \delta(1)=t^{-1}N^{-u(1+\tau)}\, ,
 \end{equation*}
 let
\begin{equation} \label{thickened_hyperplane level 1}
L^{\delta(1)}:= \left\{ \x \in I' : \dist(\x, L)=\inf_{(a_{1}, \dots , a_{d}) \in L}\left\{\max_{1 \leq i \leq d} |x_{i}-a_{i}| \right\} \leq \delta(1) \right\}\, .
\end{equation}
Then, denoting $\cS_{\ell(1)}$ as the first level surviving set of cubes in $C_{\ell(1)}$, we have
\begin{equation*}
\cS_{\ell(1)}:= \left\{ I \in C_{\ell(1)} :\, \,  I \cap \bigcup_{L \in \cL_{1}} L^{\delta(1)} = \emptyset  \right\}.
\end{equation*}
\item \textit{Iterative removal of hyperplanes}: We follow the same step as the construction of the first layer of surviving cubes, with the small exception that we only remove hyperplanes that still contain 'significant' rational points after the previous level pruning. That is, at layer
 \begin{equation} \label{l(n)}
 \ell(n):= \left\lfloor (n-1+u)\frac{(1+\tau)d}{d+1} \right\rfloor
 \end{equation}
 remove all cubes $I \in C_{\ell(n)}$ that intersect the $t^{-1}N^{-(n-1+u)(1+\tau)}$-thickening of all hyperplanes contained in $\cL_{n}^{*}$, where
 \begin{equation*}
     \cL_{n}^{*}:=\left\{ L \in \cL_{n} : \, \, 
      \exists \frac{p}{q} \in L \text{  with  } \begin{cases} i) \quad N^{n-1}\leq q<N^{n}\, ,\\
     ii)\quad   B\left(\tfrac{p}{q},\delta(n)\right)\cap \bigcup\limits_{I\in \cS_{\ell(n-1)}}I \neq \emptyset\, .
     \end{cases} \right\}.     
 \end{equation*}
 
 That is, for a hyperplane $L \in \cL_{n}^{*}$ and
 \begin{equation*}
     \delta(n)=t^{-1}N^{-(n-1+u)(1+\tau)}\, ,
 \end{equation*}
 let
\begin{equation} \label{thickened_hyperplane}
L^{\delta(n)}:= \left\{ \x \in I' : \dist(\x, L)=\inf_{(a_{1}, \dots , a_{d}) \in L}\left\{\max_{1 \leq i \leq d} |x_{i}-a_{i}| \right\} \leq \delta(n) \right\}\, .
\end{equation}
Then, denoting $\cS_{\ell(n)}$ as the surviving set of cubes in $C_{\ell(n)}$, we have
\begin{equation*}
\cS_{\ell(n)}:= \left\{ I \in C_{\ell(n)} :\, \,  I \cap \bigcup_{L \in \cL_{n}^{*}} L^{\delta(n)} = \emptyset  \quad \text{ \& } \quad \exists \, \, I' \in \cS_{\ell(n-1)} \text{ s.t. } I \subset I'\right\} \, .
\end{equation*}
For ease of notation in later stages let
\begin{equation*}
    \widehat{L^{\delta(n)}}:= \left\{ I\in C_{\ell(n)}: I\cap L^{\delta(n)} \neq \emptyset\right\}\, ,
\end{equation*}
so in particular 
\begin{equation*}
    \cS_{\ell(n)}=\left(\bigcup_{I\in \cS_{\ell(n-1)}}C_{\ell(n)}(I) \right) \backslash \left( \bigcup_{L\in\cL^{*}_{n}} \widehat{L^{\delta(n)}}\right).
\end{equation*}
For any cube $I \in C_{j}$ with $j=1, \dots , \ell(n)-1$ let $\cS_{\ell(n)}(I)$ denote the set of cubes in $\cS_{\ell(n)}$ that are contained in $I$. 
 \item Let
 \begin{equation*}
 \cC^{\tau}(N)=\bigcap_{n \in \N} \bigcup_{I \in \cS_{\ell(n)}} I.
 \end{equation*}
  \end{enumerate}

  Observe that our constructed set $\cC^{\tau}(N)$ does indeed avoid all dangerous balls. We have the following statement
  \begin{lemma}
  \begin{equation*}
  \cC^{\tau}(N) \cap \bigcup_{\frac{p}{q} \in \Q^{d} \cap [0,1]^{d}} B\left( \frac{p}{q}, c_{N}q^{-(1 + \tau)} \right)= \emptyset \, .
  \end{equation*}
  \end{lemma}
  \begin{proof}
      Suppose there exists $\frac{p}{q}\in\Q^{d}$ such that
      \begin{equation*}
          B\left(\frac{p}{q},c_{N}q^{-(1+\tau)}\right)\cap \cC^{\tau}(N) \neq \emptyset.
      \end{equation*}
      We show this to be false. Suppose $N^{k-1}\leq q < N^{k}$ for $k\in\N$. Then
      \begin{equation}\label{delta size relative to q}
          c_{N}q^{-(1+\tau)}\leq c_{N}N^{-(k-1)(1+\tau)}=t^{-1}N^{-(k-1+u)(1+\tau)}=\delta(k)\, .
      \end{equation}
      By the Simplex Lemma there exists hyperplane $L=L\left(\tfrac{p}{q}\right)\in\cL_{k}$ containing $\frac{p}{q}$. If $L\left(\tfrac{p}{q}\right)\in \cL^{*}_{k}$ then by construction we have that 
      \begin{equation*}
          L\left(\tfrac{p}{q}\right)^{\delta(n)}\cap \cC^{\tau}(N)=\emptyset \, ,
      \end{equation*}
      and so by \eqref{delta size relative to q}
      \begin{equation*}
          B\left(\frac{p}{q},c_{N}q^{-(1+\tau)}\right)\cap \cC^{\tau}(N) \subseteq B\left(\frac{p}{q},\delta(k)\right)\cap \cC^{\tau}(N) = \emptyset.
      \end{equation*}
      Thus we must have $L\left(\tfrac{p}{q}\right) \not \in \cL^{*}_{k}$. We know that $L\left(\tfrac{p}{q}\right)$ contains a rational with denominator satisfying $i)$ (the condition appearing in $\cL^{*}_{k}$), thus since $L\left(\tfrac{p}{q}\right) \not \in \cL^{*}_{k}$ it must hold that $ii)$ fails. That is
      \begin{equation*}
          B\left(\frac{p}{q},\delta(k)\right)\cap \bigcup_{I\in \cS_{\ell(k-1)}}I = \emptyset\, .
      \end{equation*}
      Since $\cC^{\tau}(N)$ is a nested Cantor set we have that
      \begin{equation*}
         \bigcup_{I\in \cS_{\ell(k-1)}}I \supseteq \cC^{\tau}(N)\, , 
      \end{equation*}
      and so, again, we are forced to conclude that
      \begin{equation*}
          B\left(\frac{p}{q},\delta(k)\right)\cap \cC^{\tau}(N) = \emptyset.
      \end{equation*}
      Appealing to \eqref{delta size relative to q} we obtain that $B\left(\tfrac{p}{q},c_{N}q^{-(1+\tau)}\right)$ has empty intersection with $\cC^{\tau}(N)$. This exhausts all possibilities for $\tfrac{p}{q}$ and so contradicts our initial assumption, thus $\cC^{\tau}(N)$ does indeed avoid all dangerous neighbourhoods of rational points.
  \end{proof}

%
%

\subsection{Lebesgue measure of \texorpdfstring{$\cC^{\tau}(N)$}{the first Cantor set}}\label{cC(N)_measure}
 The following measure theoretic properties on $\cC^{\tau}(N)$ will be required later. The first of these statements tells us that for large enough $N \in \N$ a significant proportion of $[0,1]^{d}$ is contained in $\cC^{\tau}(N)$.
 
 \begin{lemma} \label{measure_cC} 
 Suppose $\tau>\frac{1}{d}$. Then for any $\varepsilon>0$ there exists $N \in \N$ such that 
 \begin{equation*}
 \lambda_{d}(\cC^{\tau}(N)) \geq 1-\varepsilon.
\end{equation*}
Precisely,
\begin{equation*}
    \lambda_{d}(\cC^{\tau}(N)) \geq 1- 3t^{-d}\sum_{n=1}^{\infty} N^{\frac{(d+1)}{d}(n-\ell(n))}.
\end{equation*}
 \end{lemma}
 
 \begin{proof}
 Observe that 
 \begin{equation*}
 \lambda_{d}(\cC^{\tau}(N)) \geq 1-\lambda_{d}\left(\bigcup_{n \in \N}\left\{ I \in C_{\ell(n)}: I \cap  \bigcup_{L \in \cL_{n}} L^{\delta(n)} \neq \emptyset \right\} \right).
 \end{equation*}
 While we may be removing significantly less cubes (recall in later layers we only neighbourhoods of hyperplanes in the reduced set $\cL_{n}^{*}$) the calculation below is sufficient. \par 
Note that $L$ is a hyperplane and the fact that the hyperplane is contained in a cube $I \in C_{n}$ of side length $t^{-1} N^{-n\frac{d+1}{d}}$. Hence the thickened hyperplane $L^{\delta(n)}$ (recall \eqref{thickened_hyperplane}) intersects at most
 \begin{equation*}
 3N^{(\ell(n)-n)\frac{(d+1)(d-1)}{d}}
 \end{equation*}
 cubes $I \in C_{\ell(n)}$. Hence
 \begin{align} \label{thick_hyperplanes_measure}
 \lambda_{d}\left(\bigcup_{n \in \N}\left\{ I \in C_{\ell(n)}: I \cap  \bigcup_{L \in \cL_{n}} L^{\delta(n)} \neq \emptyset \right\} \right) & \leq \sum_{n=1}^{\infty} \sum_{L \in \cL_{n}} \lambda_{d}\left( \left\{ I \in C_{\ell(n)}: I \cap L^{\delta(n)} \neq \emptyset \right\}\right), \nonumber \\
 & \leq \sum_{n=1}^{\infty} N^{n(d+1)}3N^{(\ell(n)-n)\frac{(d+1)(d-1)}{d}} \lambda_{d}(I), \nonumber \\
 & \leq 3t^{-d}\sum_{n=1}^{\infty} N^{n(d+1)+(\ell(n)-n)\frac{(d+1)(d-1)}{d}-\ell(n)(d+1)}, \nonumber \\
 & \leq 3t^{-d}\sum_{n=1}^{\infty} N^{\frac{(d+1)}{d}(n-\ell(n))} \, = \varepsilon_{N}.
 \end{align}
 Thus, providing $\tau>\frac{1}{d}$ (and so $\ell(n)-n>\rho n$ for some $\rho>0$), the above summation is convergent, and for a suitably large choice of $N$ we have that
 \begin{equation*}
 \lambda_{d}(\cC^{\tau}(N)) \geq 1-\varepsilon_{N}.
 \end{equation*}
 with $\varepsilon_{N} \to 0$ as $N \to \infty$.
 \end{proof}
 
 The following lemma gives us an even stronger statement than Lemma~\ref{measure_cC}, namely that every surviving cube in the construction $\cC^{\tau}(N)$ retains much of its mass when later layers are removed.
 
 \begin{lemma} \label{measure_cC_cubes}
 Suppose $\tau >\frac{1}{d}$. Let $n, N\in \N$. Then for all $I \in \cS_{\ell(n)}$, 
 \begin{equation*}
      \lambda_{d}(I \cap \cC^{\tau}(N)) \geq \left(1-3\left( \sum_{k=n+1}^{\ell(n)}N^{(\ell(n)-\ell(k))\frac{d+1}{d}} + \sum_{k=\ell(n)+1}^{\infty} N^{(k-\ell(k))\frac{d+1}{d}} \right) \right) \lambda_{d}(I).
 \end{equation*}
 \end{lemma}
 \begin{remark} \rm Note the constant bound given here is larger than the constant proven in Lemma~\ref{measure_cC} for all $n \in \N$. So we could take $\varepsilon_{N}$ (given by \eqref{thick_hyperplanes_measure}) to be the universal constant for all surviving cube $I \in \bigcup _{n \in \N} \cS_{\ell(n)}$.
 \end{remark}
 
 \begin{proof} The proof is similar to that of Lemma~\ref{measure_cC}. Firstly, since $I \in \cS_{\ell(n)}$ we clearly have that
 \begin{equation*}
 I \not \in \bigcup_{k=1}^{n} \left\{ I' \in C_{\ell(k)}: I' \cap \bigcup_{L \in \cL_{k}} L^{\delta(k)} \neq \emptyset \right\}.
 \end{equation*}
 and so
 \begin{equation*}
 \lambda_{d}\left(I \cap \bigcup_{k=1}^{n} \bigcup_{I' \in \cS_{\ell(k)}} I'\right)=\lambda_{d}(I).
 \end{equation*}
 Hence
 \begin{equation*}
 \lambda_{d}(I \cap \cC^{\tau}(N)) \geq \lambda_{d}(I) - \lambda_{d}\left( \bigcup_{k=n}^\infty \left\{ I' \in C_{\ell(k)}(I):I' \cap \bigcup_{L \in \cL_{k}} L^{\delta(k)} \neq \emptyset \right\} \right).
 \end{equation*}
 To count the number of hyperplanes from $\bigcup_{k >n}\bigcup_{L \in \cL_{k}} L$ contained in $I$ observe that
 \begin{equation} \label{hyperplanes_1}
 \#\left\{L \in \cL_{k}: I \cap L^{\delta(k)} \neq \emptyset \right\} \leq 1 \quad \quad n+1 \leq k \leq \ell(n),
 \end{equation}
since $I$ intersects one unique cube $I' \in C_{k}$ for $n+1 \leq k \leq \ell(n)$ (in this case $I$ is actually contained in such cube). And that 
\begin{equation} \label{hyperplanes_2}
 \#\left\{L \in \cL_{k}: I \cap L^{\delta(k)} \neq \emptyset \right\} \leq N^{(k-\ell(n))(d+1)} \quad \quad k \geq \ell(n)+1,
 \end{equation}
 since we are now looking at the number of cubes $I' \in C_{k}(I)$ that generate hyperplanes. \par 
 Now we consider the number of subcubes $I' \in C_{\ell(k)}(I)$ that could intersect with a thickened hyperplane $L \in \cL_{k}$. We have that
 \begin{equation} \label{hyperplane_interesect_1}
 \#\{I' \in C_{\ell(k)}(I) : I' \cap L^{\delta(k)} \neq \emptyset\} \leq 3N^{(\ell(k)-\ell(n))\frac{(d+1)(d-1)}{d}} \quad \quad \text{ if } L \in \cL_{k} \text{ for } n+1 \leq k \leq \ell(n),
 \end{equation}
 since there are at most $N^{(\ell(k)-\ell(n))(d+1)}$ $\ell(k)$-level cubes in $I$. Also,
 \begin{equation} \label{hyperplane_intersect_2}
 \#\{I' \in C_{\ell(k)}(I) : I' \cap L^{\delta(k)} \neq \emptyset\} \leq 3N^{(\ell(k)-k)\frac{(d+1)(d-1)}{d}} \quad \quad \text{ if $L \in \cL_{k}$ for $k \geq \ell(n)+1$},
 \end{equation}
 since each hyperplane of interest is now contained in some cube in $C_{k}$. \par 
 Bring these together we have that
 \begin{align*}
 \lambda_{d}\left( \bigcup_{k \in \N_{> n}} \left\{ I' \in C_{\ell(k)}(I):I' \cap \bigcup_{L \in \cL_{k}} L^{\delta(k)} \neq \emptyset \right\} \right) & \\
 & \hspace{-5cm} \leq \sum_{k=n+1}^{\infty} \, \, \,  \sum_{L \in \cL_{k}:L \cap I \neq \emptyset} \lambda_{d}\left( \left\{ I' \in C_{\ell(k)}(I) : I' \cap L^{\delta(k)} \neq \emptyset \right\}\right), \\
 & \hspace{-6cm} \overset{\eqref{hyperplanes_1}-\eqref{hyperplane_intersect_2}}{\leq} \sum_{k=n+1}^{\ell(n)} 3t^{-d}N^{(\ell(k)-\ell(n))\frac{(d+1)(d-1)}{d}} N^{-\ell(k)(d+1)} \quad  + \\
 & \hspace{-2cm} \sum_{k=\ell(n)+1}^{\infty} 3t^{-d}N^{(k-\ell(n))(d+1)} N^{(\ell(k)-k)\frac{(d+1)(d-1)}{d}} N^{-\ell(k)(d+1)} \\
 & \hspace{-7cm}   = t^{-d}N^{-\ell(n)(d+1)}3\left( \sum_{k=n+1}^{\ell(n)}N^{(\ell(n)-\ell(k))\frac{d+1}{d}} + \sum_{k=\ell(n)+1}^{\infty} N^{(k-\ell(k))\frac{d+1}{d}} \right).
 \end{align*}
 And so we have that for $I \in \cS_{\ell(n)}$
 \begin{equation*}
 \lambda_{d}(I \cap \cC^{\tau}(N)) \geq \left( 1-3\left( \sum_{k=n+1}^{\ell(n)}N^{(\ell(n)-\ell(k))\frac{d+1}{d}} + \sum_{k=\ell(n)+1}^{\infty} N^{(k-\ell(k))\frac{d+1}{d}} \right) \right)\lambda_{d}(I).
 \end{equation*}
 
 \end{proof}

%
%

\section{Step 2: A suitable subset of \texorpdfstring{$\W_{d}(\psi_{\tau})$}{well approximable points}} \label{subset_W(psi)}

Recall we want to construct a subset of $\W_{d}(\psi_{\tau})$ such that any ball in the corresponding $\limsup$ set retains enough mass when intersected with $\cC^{\tau}(N)$. To this end we construct the following subset of $\Q^{d}$. 
Define
\begin{equation*}
    Q(N,\tau):=\left\{ \frac{p}{q} \in \Q^{d}: \exists \, L\in\cL^*_{n} \, \text{ with } \,  \frac{p}{q}\in L \quad \text{ and } \begin{cases} i)\,  N^{n-1} \leq q \leq N^{n} \, , \\
    ii) \, B\left(\tfrac{p}{q},\delta(n)\right)\cap \bigcup_{I\in \cS_{\ell(n-1)}}I \neq \emptyset\, .
    \end{cases}
    \right\}\, .
\end{equation*}
We say $\tfrac{p}{q}$ is a \textit{leading rational of a hyperplane $L \in \cL_{n}$} if $L$ is the hyperplane on which $\tfrac{p}{q}$ lies while satisfying the above conditions. \par
Essentially, for each hyperplane removed in the construction of $\cC^{\tau}(N)$ we associate a set of rational points $\tfrac{p}{q}$ and significant intersection with surviving cubes of the previous layer. \par 
Note a few important properties of leading rationals:
\begin{enumerate}
    \item Every hyperplane $L\in \bigcup_{n\in\N}\cL^{*}_{n}$ contains at least one leading rational. This is clear from how we define the sets $\cL^{*}_{n}$ and $Q(N,\tau)$. Namely, if $L$ does not contain a leading rational, then it would not have been removed.
    \item If $\tfrac{p}{q}$ is a leading rational of a hyperplane $L \in \cL^{*}_{n}$ that was constructed in cube $I\in C_{n}$ then
    \begin{equation*}
        B\left(\tfrac{p}{q}, 2q^{-1-\frac{1}{d}}\right) \supseteq I\, .
    \end{equation*}
    To see this observe that trivially $\tfrac{p}{q} \in I$, that the sidelenght of $I$ is $N^{-n(1+\tfrac{1}{d})}$, and that $N^{n-1}\leq q \leq N^{n}$.
    \item If $\tfrac{p}{q}$ is not a leading rational then there exists a hyperplane $L\in \bigcup_{n\in\N}\cL^{*}_{k}$, say $L\in \cL^{*}_{n}$, for which $\tfrac{p}{q} \in \widehat{L^{\delta(k)}}$ and any leading rational of $L$, say $\tfrac{r}{s}$, has $s<q$. This may not be so obvious. To see this, suppose $N^{n-1}< q \leq N^{n}$. Then $\tfrac{p}{q}$ must, by the simplex lemma, lie on a hyperplane of $\cL_{n}$. If $\tfrac{p}{q}$ is not a leading rational of said hyperplane, then it must lie in some previous level removed strip. Any leading rational, say $\frac{r}{s}$, of a hyperplane from a previously removed layer must, by definition, have denominator $s<N^{n-1}<q$. 
\end{enumerate}

 Let 
\begin{equation*}
\W_{d}(Q(N,\tau), \psi):= \left\{ \x \in [0,1]^{d} : \max_{1 \leq i \leq d}\left| x_{i}-\frac{p_{i}}{q} \right| < \frac{\psi(q)}{q} \quad \text{ for i.m. } \, \, \frac{p}{q} \in Q(N,\tau) \right\}.
\end{equation*}
The following two lemmas are crucial properties of the set $Q(N,\tau)$. 

\begin{lemma} \label{Q_approximations}
Let $\psi(q)=3q^{-\frac{1}{d}}$ and suppose that $\tau>\frac{1}{d}$. Then
\begin{equation*}
\W_{d}(Q(N,\tau), \psi)\supseteq \cC^{\tau}(N)\, .
\end{equation*}
\end{lemma}

\begin{proof}
Take any $\x \in \cC^{\tau}(N)$. By Dirichlet's Theorem $\x \in \W_{d}(\psi)$. Let $A(\x) $ 
be the set of all those rational points for which
\begin{equation*}
\x \in B\left(\frac{p}{q}, q^{-1-\frac{1}{d}}\right).
\end{equation*}
If the cardinality of $A(\x) \cap Q(N, \tau)$ is infinite, then clearly $\x \in \W_{d}(Q(N,\tau), \psi)$, so assume $A(\x) \cap Q(N,\tau)$ is finite. Let $\frac{p}{q} \in A(\x) \backslash Q(N,\tau)$. Let $N^{m-1}<q\le N^m$. Since $\tfrac{p}{q}$ is not a leading rational, by $3.$ there exists some hyperplane $L$, say $L\in\cL^{*}_{n}$ with $n<m$, such that $\tfrac{p}{q} \in \widehat{L^{\delta(n)}}$. Let $\tfrac{r}{s}$ be a leading rational of $L$. Then by property $3.$ of leading rational $s<q$, and by the triangle inequality
\begin{equation} \label{eq1}
    \left|\x-\frac{r}{s}\right| \leq \left|\x-\frac{p}{q}\right|+ \left|\frac{p}{q}-\frac{r}{s}\right|< 3s^{-1-\tfrac{1}{d}}.
\end{equation}
If the sequence of rational points in $A(\x)$ can be associated to infinitely many different leading rationals, then we are done by the above argument. We now prove that there must be hyperplanes $L\in \cL_n^*$ from arbitrarily high levels $n$ associated to the points in the sequence $A(\x)$. That is, they are leading rationals on these hyperplanes themselves, or we find a nearby leading rational by the above argument.

Assume to the contrary that the sequence $A(\x)$ is associated to finitely many hyperplanes. Suppose $m$ is the largest level in which this finite sequence of hyperplanes appear. By definition of $\x \in \cC^{\tau}(N)$ we have that
\begin{equation*}
    \x \in [0,1]^{d} \backslash \left( \bigcup_{j\leq m} \bigcup_{L\in \cL^{*}_{j}} \widehat{L^{\delta(j)}}\right)\, ,
\end{equation*}
which is an open set, and so there exists some $\varepsilon>0$ such that
\begin{equation*}
    d\left(\x, \bigcup_{j\leq m} \bigcup_{L\in \cL^{*}_{j}} \widehat{L^{\delta(j)}}\right)>\varepsilon\, .
\end{equation*}
However, for any $\frac{u}{v}\in A(\x)$ with $v>\varepsilon^{-1}$ we have that
\begin{equation*}
\left| \x -\frac{u}{v}\right|<v^{-1-\tfrac{1}{d}}<\varepsilon\, ,
\end{equation*}
and so
\begin{equation*}
    \frac{u}{v} \not \in \bigcup_{j\leq m} \bigcup_{L\in \cL^{*}_{j}} \widehat{L^{\delta(j)}}\, ,
\end{equation*}
contradicting the finiteness of the sequence of associated hyperplanes. Thus we have an infinite sequence of associated hyperplanes to $A(\x)$, which by \eqref{eq1} we can associate to an infinite sequence of rational points which sufficiently approximate $\x$ so that $\x \in \W_{d}(Q(N,\tau),\psi)$ as required.

\end{proof}

Lastly, the following lemma on the measure of balls in the construction of $\W_{d}(Q(N,\tau),\psi_{\tau})$ will be required later. It essentially tells us that each ball $B\left(\frac{p}{q},q^{-1-\tau}\right)$ with $\frac{p}{q}\in Q(N,\tau)$ contains a surviving cube that covers a constant proportion of $B\left(\frac{p}{q},q^{-1-\tau}\right)$ independent of $\frac{p}{q}$.

  \begin{lemma} \label{I in B(t)}
Let $B(\tau)=B\left(\tfrac{p}{q},q^{-1-\tau}\right)$ for some $\frac{p}{q}\in Q(N,\tau)$. Then there exists $I \in \bigcup_{n\in \N} \cS_{\ell(n)}$ with $I \subset B(\tau)$ and
\begin{equation*}
    \lambda_{d}(B(\tau))\leq C\lambda_{d}(I),
\end{equation*}
with $C>0$ independent of $\frac{p}{q}$.
\end{lemma}

\begin{proof}
    Suppose that $N^{n-1}\leq q< N^{n}$. Since $\frac{p}{q} \in Q(N,\tau)$ we have that
    \begin{equation*}
        B\left(\frac{p}{q}, \delta(n)\right)\cap \bigcup_{I\in \cS_{\ell(n-1)}}I \neq \emptyset\, .
    \end{equation*}
    Let $I^{*}\in \cS_{\ell(n-1)}$ be some cube with positive intersection with $B\left(\tfrac{p}{q},\delta(n)\right)$. Then, by considering the sidelength of $I^{*}$ we have that
    \begin{equation*}
        B\left(\frac{p}{q}, 3\max\left\{t^{-1}N^{-\ell(n-1)\left(1+\tfrac{1}{d}\right)},\delta(n)\right\} \right) \supseteq I^{*}\, .
    \end{equation*}
    Now observe that
    \begin{align*}
        3t^{-1}N^{-\ell(n-1)\left(1+\frac{1}{d}\right)} &\leq 3t^{-1}N^{-\left\lfloor(n-2+u)\frac{(1+\tau)d}{(d+1)}\right\rfloor\frac{d+1}{d}} \\
        &\leq 3t^{-1}N^{-(n-1+u)\frac{(1+\tau)d}{(d+1)}\frac{d+1}{d}}\,\,\, \,\,\,(=3\delta(n))\\
        &=3t^{-1}N^{-(u-1)(1+\tau)}N^{-n(1+\tau)}\\
        &<q^{-(1+\tau)}\, .
    \end{align*}
    Hence 
    \begin{equation*}
        B\left(\frac{p}{q},q^{-(1+\tau)}\right) \supseteq B\left(\frac{p}{q}, 3\max\left\{t^{-1}N^{-\ell(n-1)(1+\tfrac{1}{d})},\delta(n)\right\} \right) \supseteq I^{*}\, .
    \end{equation*}
    It remains to see that 
    \begin{align*}
        \lambda_{d}(I^{*})&= t^{-d}N^{-\ell(n)(d+1)}\\
        &\geq t^{-d}N^{-(n-2+u)(1+\tau)d}\\
        & \geq t^{-d}N^{-(u-1)(1+\tau)d}q^{-(1+\tau)d}\\
        &=2^{-d}t^{-d}N^{-(u-1)(1+\tau)d} \lambda_{d}\left(B\left(\frac{p}{q},q^{-(1+\tau)}\right)\right)\, ,
    \end{align*}
    and so taking $C=2^{d}t^{d}N^{(u-1)(1+\tau)d}$ in the lemma completes the proof.
    \end{proof}

%
%

\section{Step 3: Construction of a Cantor subset of \texorpdfstring{$\Bad_{d}(\psi_{\tau})$}{certain badly approximable sets}} \label{cD(B)}

The following lemma is crucial in the construction of our Cantor subset of $\Bad_{d}(\psi_{\tau})$.

\begin{lemma}[$T_{G,I}$ Lemma] \label{TGI}
Let $\{B_{i}\}_{i \in \N}$ be a sequence of balls in $\R^{d}$ with $r(B_{i}) \to 0$ as $i \to \infty$. Suppose that there exists some constant $C>0$ such that 
\begin{equation} \label{partial_measure}
\lambda_{d}\left( I \cap \limsup_{i \to \infty} B_{i}\right)\geq C \lambda_{d}(I)
\end{equation}
for any $I \in \bigcup_{n \in \N} \cS_{\ell(n)}$. Then for any $I \in \bigcup_{n \in \N} \cS_{\ell(n)}$ and any $G>1$ there is a finite subcollection
\begin{equation*}
T_{G,I} \subset \{B_{i}: i \geq G\}
\end{equation*}
such that the balls are disjoint, lie inside $I$, and 
\begin{equation*}
\lambda_{d}\left( \bigcup_{\tilde{B} \in T_{G,I}}\tilde{B}\right) \geq \kappa_{1} \lambda_{d}(I),
\end{equation*}
with $\kappa_{1}=C5^{-d}4^{-1}$.
\end{lemma}

\begin{remark} \rm 
Note that Lemma~\ref{TGI} is applicable to our setting with
\begin{equation*}
    \{B_{i}\}_{i \in \N}=\left\{ B\left(\frac{p}{q}, 2q^{-1-\frac{1}{d}} \right) \right\}_{\frac{p}{q} \in Q(N,\tau)}\, ,
\end{equation*}
i.e. $\limsup_{i \to \infty} B_{i}=\W_{d}\left(Q(N, \tau), 2 \psi_{1/d}\right)$, since
\begin{align*}
    \lambda_{d}\left(I \cap \W_{d}\left(Q(N, \tau), 2\psi_{1/d}\right) \right) & \overset{\text{ Lemma~\ref{Q_approximations}}}{\geq} \lambda_{d}\left(I \cap \cC^{\tau}(N) \cap \W_{d}(\psi_{1/d}) \right) \\
    & = \lambda_{d}\left(I \cap \cC^{\tau}(N) \right)\\
    & \overset{\text{ Lemma~\ref{measure_cC_cubes}}}{\geq} (1-\varepsilon_{N})\lambda_{d}(I).
\end{align*}
Note that we can order the collection of balls in a decreasing order. 
\end{remark}

\begin{proof}
Suppose that $I \in \cS_{\ell(n)}$ and let $\cS_{\ell(n+1)}^{*}(I)$ denote the set of cubes $I' \in \cS_{\ell(n+1)}$ contained in the \textbf{interior} of $I$ (remove the edge cubes). Observe that
\begin{equation} \label{eq1.1}
\lambda_{d}\left(\bigcup_{I' \in \cS_{\ell(n+1)}^{*}(I)} I' \right) \geq \frac{1}{2}\lambda_{d}(I)
\end{equation}
 since we are removing at most $2^{d}+1$ hyperplanes worth of cubes which (for sufficiently large $N \in \N$) is small relative to the size of $I$. \par 
 Let
\begin{equation*}
\cG:=\left\{B_{i} : B_{i} \cap \bigcup_{I' \in \cS_{\ell(n+1)}^{*}(I)} I' \neq \emptyset\, , \, \, i \geq G \right\}.
\end{equation*}
We may assume $r(B_{i})$ is monotonically decreasing (since we can order $B_{i}$ however we want) and so for sufficiently large $i$ ($i \geq i_{0}$ such that $r(B_{i_{0}}) < \frac{1}{2}N^{-\ell(n+1)\frac{d+1}{d}}$) any ball $B_{i} \in \cG$ is contained in $I$. By the $5r$-covering lemma (see for example \cite[Theorem 1.2]{Hein01} ) there exists a disjoint subcollection $\cG' \subseteq \cG$ such that
\begin{equation*}
\bigcup_{B_{i} \in\cG} B_{i} \subset \bigcup_{B_{i} \in \cG'}5B_{i}
\end{equation*}
and the balls in $\cG'$ are disjoint. It follows that
\begin{align*}
\lambda_{d}\left( \bigcup_{B_{i} \in \cG'} 5B_{i} \right) & \geq \lambda_{d}\left( \bigcup_{I' \in \cS_{\ell(n+1)}^{*}(I)}I' \cap \limsup_{i \to \infty} B_{i} \right) \\
& =\sum_{I' \in \cS_{\ell(n+1)}^{*}(I)} \lambda_{d}\left( I' \cap \limsup_{i \to \infty} B_{i} \right) \\
& \overset{\eqref{partial_measure}}{\geq} C \sum_{I' \in \cS_{\ell(n+1)}^{*}(I)} \lambda_{d}(I') \\
& \overset{\eqref{eq1.1}}{\geq} \frac{C}{2} \lambda_{d}(I).
\end{align*}
Furthermore, since $\cG'$ is a disjoint collection of balls we have that
\begin{align*}
\lambda_{d}\left( \bigcup_{B_{i} \in \cG'} 5B_{i} \right) & \leq \sum_{B_{i} \in \cG'} 5^{d} \lambda_{d}(B_{i}) \\
& = 5^{d}\lambda_{d}\left( \bigcup_{B_{i} \in \cG'} B_{i} \right).
\end{align*}

Hence
\begin{equation*}
\lambda_{d}\left( \bigcup_{B_{i} \in \cG'} B_{i} \right) \geq \frac{C}{5^{d}2} \lambda_{d}(I).
\end{equation*}
Since the balls of $\cG'$ are disjoint and $r(B_{i}) \to 0$ as $i \to \infty$ we have that
\begin{equation*}
\lambda_{d}\left( \bigcup_{B_{i} \in \cG': i \geq j} B_{i} \right) \to 0 \quad \text{ as } \quad j \to \infty.
\end{equation*}
Hence there exists $j_{0} \geq G$ such that
\begin{equation*}
\lambda_{d}\left( \bigcup_{B_{i} \in \cG': i \leq j_{0}} B_{i} \right) \geq \frac{C}{5^{d}4} \lambda_{d}(I).
\end{equation*}
Letting 
\begin{equation*}
T_{G,I}=\{B_{i} \in \cG' : i \leq j_{0}\}
\end{equation*}
completes the proof.
\end{proof}

\subsection{Constructing \texorpdfstring{$\cD(B)$}{the second Cantor set}}


Armed with Lemma~\ref{TGI} we now proceed with the construction of a Cantor subset $\cD(B)$ of 
\begin{equation*}
\cC^{\tau}(N) \cap W_{d}(Q(N,\tau), \psi_{\tau}) \cap B \, 
\end{equation*}
for any ball $B \subset [0,1]^{d}$. A sketch of the construction is as follows:
\begin{enumerate}[i)]
    \item Firstly given a ball $B$ we choose a suitable $N\in\N$ such that $B$ has non-empty intersection with $\cC^{\tau}(N)$. (We will choose $N$ larger later when necessary.) 
    \item To a surviving cube $I\subset B$ we then apply Lemma~\ref{TGI} to obtain a disjoint collection of balls from  $\left\{B\left(\frac{p}{q}, 2q^{-1-\frac{1}{d}} \right)\right\}_{\frac{p}{q} \in Q(N,\tau)}$ contained in $I$ that cover a positive proportion of $I$.
    \item We then shrink each of these balls i.e.
\begin{equation*}
    B\left(\frac{p}{q}, 2q^{-1-\frac{1}{d}}\right) \mapsto B\left( \frac{p}{q}, q^{-1-\tau} \right)\, .
\end{equation*}
The collection of these shrunken balls completes the first layer.
\item We repeat the same process as above. That is, for each shrunken ball from the first layer we apply Lemma~\ref{I in B(t)} to find a surviving cube $I'$ of $C^{\tau}(N)$. We then choose suitable $G'>0$ and apply Lemma~\ref{TGI} to each surviving cube $I'$, and then shrink each of the balls in the set $T_{G',I'}$. This collection of balls then becomes the second layer. We continue inductively. 
\end{enumerate}

Let us now present the full details of the construction. Fix $B\subset [0,1]^{d}$. We construct $\cD(B)$ as follows:
\begin{enumerate}
    \item Fix $N \in \N$ such that
    \begin{equation} \label{N_choice}
        N \geq r(B)^{\frac{d}{d+1}} \quad \& \quad 3t^{-d}\sum_{n=1}^{\infty} N^{\frac{(d+1)}{d}(n-\ell(n))}=\varepsilon_{N}< \frac{1}{4}\lambda_{d}(B)\, .
    \end{equation}
    This ensures that
    \begin{equation*}
        \lambda_{d}(B\cap \cC^{\tau}(N))\geq \frac{3}{4}\lambda_{d}(B),
    \end{equation*}
    Furthermore, assume $N$ is large enough such that there exists $I_{0}\in \cS_{\ell(1)}$ with $I_{0} \subset B$. Set $E_{0}=I_{0}$.\\
  
    \item Apply Lemma~\ref{TGI} to $I_{0}$ with $G_{0}\geq 1$ chose sufficiently large such that
    \begin{equation} \label{G bound}
     \left. \begin{array}{c}
     r(B_{i})<t^{-1}N^{-\ell(n_{0})\frac{d+1}{d}} \\
     4r(B_{i}(\tau))<r(B_{i}) 
     \end{array} \right\} \quad \forall \, \, i \geq G_{0}.
    \end{equation}
    Let $T_{G_{0},I_{0}}$ denote the set of balls $B_{i}$ from Lemma~\ref{TGI}. \\

    \item \textit{Shrinking:} For each $B_{i} \in T_{G_{0},I_{0}}$ shrink the balls
    \begin{equation*}
        B_{i}=B\left(\frac{p}{q},2q^{-\frac{d+1}{d}}\right) \mapsto B\left(\frac{p}{q},q^{-(1+\tau)}\right)=B_{i}(\tau),
    \end{equation*}
    and let the first layer
    \begin{equation*}
        E_{1}=T_{G_{0},I_{0}}^{\tau}=\left\{B_{i}(\tau):B_{i}\in T_{G_{0},I_{0}} \right\}.
    \end{equation*}
    \\
    \item \textit{Induction:} Suppose $E_{n-1}$ has been constructed. For each $B_{i}(\tau)\in E_{n-1}$ we proceed as follows.
    \begin{enumerate}
        \item[a)] \textit{nesting:} By Lemma~\ref{I in B(t)} there exists some $I\in \bigcup_{n}\cS_{\ell(n)}$, say $I(i)$, with $I(i)\subseteq B_{i}(\tau)$ and 
    \begin{equation*}
        \lambda_{d}(B_{i}(\tau))\leq C(N,d,u,\tau)\lambda_{d}(I(i)).
    \end{equation*}
    For $I(i)$ related to $B_{i}(\tau)$ in this way say $I(i) \sim B_{i}(\tau)$. \\

    \item[b)] \textit{covering:} Choose $G_{n}> G_{n-1}$ such that
    \begin{equation*}
        4r(B_{j})<\min\{r(I(k)): I(k)\sim B_{k}(\tau)\in E_{n-1} \} \quad \forall \, \, j\geq G_{n}.
    \end{equation*}
    (Note that $r(B_j)\to 0$ with $j$.) Again, note this is only a lower bound, $G_{n}$ can be chosen as large as we want, and indeed, we will increase it if necessary in the proof of Lemma \ref{lem:msrballs}. Apply Lemma~\ref{TGI} to $I(i)$ with $G_{n}$ chosen as above. Let $T_{G_{n},I(i)}$ be the set of balls from Lemma~\ref{TGI}. \\

    \item[c)] \textit{shrinking:} For each $B_{i} \in T_{G_{n},I(i)}$ shrink the ball and let
    \begin{equation*}
        T_{G_{n},I(i)}^{\tau}=\left\{ B_{j}(\tau): B_{j} \in T_{G_{n},I(i)} \right\}.
    \end{equation*}
    Define the $n$th layer to be
    \begin{equation*}
        E_{n}=\left\{ B_{j}(\tau) : B_{j}(\tau) \in \bigcup_{I(k)\sim B_{k}(\tau) \in E_{n-1}} T^{\tau}_{G_{n},I(k)} \right\}.
    \end{equation*} 
    \end{enumerate}
\end{enumerate}
Let
\begin{equation*}
    \cD(B)=\bigcap_{n\in \N} \bigcup_{B(\tau)\in E_{n}} B(\tau)
\end{equation*}
This completes the construction of $\cD(B)$. \\
Note that $\cD(B) \subset D_d(\psi_\tau, c_N\psi_\tau)\subset \Bad_d(\psi_\tau)$ since for any $\x \in \cD(B)$ we have an associated sequence
    \begin{equation*}
        I_{0} \supset B_{i_{1}}(\tau) \supset I(i_{1}) \supset B_{i_{2}}(\tau) \supset I(i_{2}) \dots
    \end{equation*}
such that $\x$ is contained in the limit set. So $\x \in \cC^{\tau}(N) \cap W(Q(N,\tau),\psi_{\tau}) \subset D_d(\psi_\tau, c_N\psi_\tau)\subset \Bad_{d}(\psi_{\tau})$.

\begin{remark} \rm The methodology of the proof given here follows closely to the usual proof of the Mass Transference Principle. The key additional step in the construction is the \textit{nesting} step, which we require to ensure $\cD(B)\subset \cC^{\tau}(N)$, and replacing the $K_{G,B}$-lemma \cite[Lemma 5]{BV06} by $T_{G,I}$ Lemma \ref{TGI} above. 
\end{remark}

\subsection{Constructing a measure on \texorpdfstring{$\cD(B)$}{the second Cantor set}}

To obtain a lower bound to the Hausdorff dimension of $\cD(B)$ and hence to $\Bad_d(\psi_\tau)$, we aim to use the mass distribution principle \cite[Proposition 4.2]{F14}. Namely, we need to construct a measure $\mu$ supported on $\cD(B)$ such that measures of balls are in good control. 

To that end, construct a mass distribution $\mu$ on $\cD(B)$ as follows:\\
Set $\mu(I_{0})=1$, then for each $B_{i}(\tau)\in E_{1}$ define
\begin{equation*}
    \mu(B_{i}(\tau))=\frac{\lambda_{d}(B_{i})}{\sum_{B_{j} \in T_{G_{0},I_{0}}} \lambda_{d}(B_{j})}\times \mu(I_{0}).
\end{equation*}
Note
\begin{equation*}
    \sum_{B_{i}(\tau) \in E_{1}}\mu(B_{i}(\tau))=\mu(I_{0}).
\end{equation*}
Now inductively suppose the mass has been distributed over balls in $E_{n-1}$. For each $B_{i}(\tau) \in E_{n}$ there exists, by construction, a unique $B_{k}(\tau) \in E_{n-1}$ such that $B_{i}(\tau)\subset B_{k}(\tau)$. Let $I(k)\sim B_{k}(\tau)$ and define
\begin{equation*}
    \mu(B_{i}(\tau))= \frac{\lambda_{d}(B_{i})}{\sum_{B_{j} \in T_{G_{n},I(k)}} \lambda_{d}(B_{j})}\times \mu(B_{k}(\tau)).
\end{equation*}
 Again the mass is preserved. By Proposition 1.7 of \cite{F14} this extends to a measure $\mu$ with support contained in $\cD(B)$ and for general subset $F\subset [0,1]^{d}$
 \begin{equation*}
     \mu(F)=\inf\left\{ \sum_{i}\mu(B_{i}(\tau)): F\cap \cD(B) \subset \bigcup_{i}B_{i}(\tau) \, \text{ with } \, B_{i}(\tau) \in \bigcup_{n}E_{n} \right\}.
 \end{equation*}
 Recall that $s=(1+d)/(1+\tau)$. 

\begin{lemma}\label{lem:msrballs}
Let $\epsilon>0$. Let $F$ be a ball contained in $B$. Then there is a choice of suitable $(G_n)$ in steps 4 of the construction of $\cD(B)$ such that 
\[
\mu(F)\le r(F)^{s-\epsilon}. 
\]
\end{lemma}
\begin{proof}
{\bf Case 1:  $B_{i}(\tau) \in E_{n}$.}

Assume first that $F=B_{i}(\tau) \in E_{1}$, a ball of centre $\tfrac pq\in \Q^d$ and radius $q^{-(1+\tau)}$. Then 
 \begin{align*}
     \mu(B_{i}(\tau))& =\frac{\lambda_{d}(B_{i})}{\sum_{B_{j} \in T_{G_{0},I_{0}}} \lambda_{d}(B_{j})}\times 1 \\
     & \overset{\text{Lemma~\ref{TGI}}}{\leq} \lambda_{d}(B_{i}) \frac{1}{\kappa_{1}\lambda_d(I_{0})} \\
     & \leq \frac{1}{\kappa_{1}\lambda_d(I_{0})} 2^{2d}q^{-(d+1)} \\
     &= \frac{2^{2d}}{\kappa_{1}\lambda_d(I_{0})} q^{-(1+\tau)\frac{d+1}{1+\tau}} \\
     &= \frac{2^{2d}}{\kappa_{1}\lambda_d(I_{0})} r(B_{i}(\tau))^{s}.
 \end{align*}
 For balls $B_{i}(\tau) \in E_{n}$ centred at some $\tfrac pq\in \Q^d$ with $n>1$ there exists $B_{k}(\tau) \in E_{n-1}$ with $B_{i}(\tau)\subset B_{k}(\tau)$. Let
 \begin{align*}
     \mu(B_{i}(\tau))& =\frac{\lambda_{d}(B_{i})}{\sum_{B_{j} \in T_{G_{0},I(k)}} \lambda_{d}(B_{j})}\times \mu(B_{k}(\tau)) \\
     & \overset{\text{Lemma~\ref{TGI}}}{\leq} \lambda_{d}(B_{i}) \frac{\mu(B_{k}(\tau))}{\kappa_{1}\lambda_d(I(k))} \\
     & \leq \frac{\mu(B_{k}(\tau))}{\kappa_{1}\lambda_d(I(k))} 2^{2d}q^{-(d+1)} \\
     &= \frac{2^{2d}}{\kappa_{1}} r(B_{i}(\tau))^{s} \frac{\mu(B_{k}(\tau))}{\lambda_d(I(k))}\\
     &\overset{\text{Lemma~\ref{I in B(t)}}}{\leq} \frac{2^{2d}}{\kappa_{1}C(N,d,u,\tau)} r(B_{i}(\tau))^{s} \frac{\mu(B_{k}(\tau))}{\lambda_d(B_{k}(\tau))}
 \end{align*}
The value of $G_{n}$ can be chosen as large as we like, and hence for any $\epsilon>0$ we can suppose it was large enough to guarantee
 \begin{equation}\label{assumption1}
 r(B_{i}(\tau))^{-\epsilon}>\frac{\mu(B_{k}(\tau))}{\lambda_d(B_{k}(\tau))}.
 \end{equation}
 Note that this can be done in a uniform manner since for any $B_i(\tau)\in E_n$ we have $i>G_n$ and the set $E_{n-1}$ is finite. Hence, 
 \begin{equation*}
     \mu(B_{i}(\tau))\leq C'(N,d,u,\tau) r(B_{i}(\tau))^{s-\epsilon}.
 \end{equation*}
 Observe the constant is independent of the layer which $B_{i}(\tau)$ is from.

{\bf Case 2:  $F$ general ball.}

Now fix $F$, a ball contained in $B$. We prove the general statement that for every $\epsilon>0$
 \begin{equation}\label{general ball}
     \mu(F)\ll_{N,d,u,\tau} r(F)^{s-\epsilon}.
 \end{equation}
 Firstly note that if $F\cap \cD(B)=\emptyset$ then $\mu(F)=0$ so the statement is trivially satisfied. Hence without loss of generality $F\cap \cD(B)\neq \emptyset$. Since $F$ has non-empty intersection with $\cD(B)$ at least one ball in each layer $E_{n}$ must intersect $F$. If there is exactly one ball $B_{i_{n}}(\tau)$ per layer $E_{n}$ (that intersects $F$) then note that 
 \begin{equation*}
     \mu(B_{i_{n}}(\tau))\to 0 \quad \text{ as } \, n\to \infty
 \end{equation*}
 and so \eqref{general ball} holds, since there exists some $n\in \N$ such that
 \begin{equation*}
    \mu(B_{i_{n}}(\tau))\ll r(F)^{s-\epsilon}.
 \end{equation*}
 Hence assume there exists $n_{F}\in \N$ such that at least two balls, say $B_{i}(\tau),B_{j}(\tau) \in E_{n_{F}}$ have non-empty intersection with $F$. Furthermore $n_{F}$ is the first layer in which this happens.\par 
Note $B_{i}(\tau),B_{j}(\tau)$ belong to the same ball $B_{k}(\tau) \in E_{n_{F}-1}$ (otherwise the $n_{F}-1$ layer has non-empty intersection with at least two balls, contradicting the choice of $n_{F}$). If
 \begin{equation*}
     r(F)\geq r(B_{k}(\tau)),
 \end{equation*}
 then
 \begin{equation*}
     \mu(F) \leq \underset{B_{i}(\tau)\cap F \neq \emptyset}{\sum_{B_{i}(\tau) \in E_{n_{F}-1} :}}\mu(B_{i}(\tau))=\mu(B_{k}(\tau)). 
         \end{equation*}
     We can then use Case 1 of the proof to estimate $\mu(B_k(\tau))$ from above and obtain 
     \begin{equation*}
      \mu(F) \leq C'(N,d,u,\tau) r(B_{k}(\tau))^{s-\epsilon}
     \leq C'(N,d,u,\tau) r(F)^{s-\epsilon}
 \end{equation*}
 as required. \par 
  
 So assume 
 \begin{equation}\label{assumption2}
 r(F)< r(B_{k}(\tau)).    
 \end{equation}
 We have that
 \begin{equation} \label{F count}
    \lambda_{d}(5F)\geq \sum_{B_{i}:B_{i}(\tau)\in \cF} \lambda_{d}(B_{i}),
\end{equation}
where 
\begin{equation*}
    \cF=\left\{B_{i}(\tau) \in T_{G_{n_{F}},I(k)}^{\tau}: B_{i}(\tau)\cap F \neq \emptyset\right\}.
\end{equation*}
See \cite[Lemma 7]{BV06}, setting $A=F$ and $M=B_{i}(\tau)$, $cM=B_{i}$ for justification of the above statement. Hence
 \begin{align*}
     \mu(F)&\leq \sum_{B_{i}(\tau) \in \cF} \mu(B_{i}(\tau)) \\
     & \leq \sum_{B_{i}:B_{i}(\tau) \in \cF} \frac{\lambda_{d}(B_{i})}{\sum_{B_{j}\in T_{G_{n_{F}},I(k)}}\lambda_{d}(B_{j})} \times \mu(B_{k}(\tau)) \\
     & \overset{\text{Lemma~\ref{TGI}}}{\leq} \frac{1}{\kappa_{1}}\sum_{B_{i}:B_{i}(\tau)\in \cF}\lambda_{d}(B_{i}) \times \frac{\mu(B_{i}(\tau))}{\lambda_{d}(I(k))}\\
     &\overset{\text{Lemma~\ref{I in B(t)}}}{\leq} \frac{1}{\kappa_{1}C(N,d,u,\tau)}\sum_{B_{i}:B_{i}(\tau)\in \cF}\lambda_{d}(B_{i}) \times \frac{\mu(B_{i}(\tau))}{\lambda_{d}(B_{k}(\tau))} \\
     & \overset{\eqref{F count}}{\leq} \frac{5^{d}}{\kappa_{1}C(N,d,u,\tau)} \lambda_{d}(F) \frac{\mu(B_{i}(\tau))}{\lambda_{d}(B_{k}(\tau))}. 
     \end{align*}
     Finally, recalling that $d>s$ we continue with 
     \begin{align*}
    \mu(F)&\overset{(d>s)}{\leq} \frac{10^{d}}{\kappa_{1}C(N,d,u,\tau)} r(F)^{s} \frac{\mu(B_{i}(\tau))}{\lambda_{d}(B_{k}(\tau))}\\
     &\leq \frac{10^{d}}{\kappa_{1}C(N,d,u,\tau)} r(F)^{s-\epsilon},
 \end{align*}
so \eqref{general ball} holds, where the last line follows on the observation that
\begin{equation*}
    r(F)^{-\epsilon}\overset{\eqref{assumption2}}{>}r(B_{k}(\tau))^{-\epsilon}\overset{\eqref{assumption1}}{>} \frac{\mu(B_{i}(\tau))}{\lambda_{d}(B_{k}(\tau))}.
\end{equation*}

\end{proof}

We can now finish the proof of Theorem \ref{D-set}. By the mass distribution principle, see for example \cite[Proposition 4.2]{F14}, we have that
\begin{equation*}
    \dimh \cD(B) \geq s-\varepsilon.
\end{equation*}
Since $\varepsilon>0$ is arbitrary the lower bound dimension result of Theorem~\ref{D-set}, and thus Theorem~\ref{main}, follows.
\bibliographystyle{plain}
\bibliography{biblio}

\end{document}